\documentclass[12pt]{amsart}
\usepackage{amsmath,amssymb}

\newcommand{\IZ}{{\mathbf Z}}

\newcommand{\iip}[1]{\mathopen{\langle\!\langle}#1\mathclose{\rangle\!\rangle}}
\newcommand{\wt}{w}
\DeclareMathOperator{\spann}{span}
\DeclareMathOperator{\GL}{GL}
\DeclareMathOperator{\EL}{E}

\newtheorem{thm}{Theorem}
\newtheorem{cor}[thm]{Corollary}
\newtheorem{lem}[thm]{Lemma}

\theoremstyle{definition}
\newtheorem*{prob}{Problem}
\newtheorem*{rem}{Remark}

\title{Is an irng singly generated as an ideal?}
\author{Nicolas Monod}
\address{Nicolas Monod, EPFL, 1015 Lausanne, Switzerland}
\email{nicolas.monod@epfl.ch}
\author{Narutaka Ozawa}
\address{Narutaka Ozawa, RIMS, Kyoto University, \mbox{606-8502}, Japan}
\email{narutaka@kurims.kyoto-u.ac.jp}
\author{Andreas Thom}
\address{Andreas Thom, Univ.\ Leipzig,
PF 100920, 04009 Leipzig , Germany}
\email{andreas.thom@math.uni-leipzig.de}

\thanks{N.M. was partially supported by the SNF and the ERC. N.O. was partially supported by JSPS and Sumitomo Foundation. A.T. was supported by the ERC Starting Grant 277728.}


\begin{document}
\begin{abstract}
Recall that a rng is a ring which is possibly non-unital.
In this note, we address the problem whether every finitely generated
idempotent rng (abbreviated as irng) is singly generated as an ideal.
It is well-known that it is the case for a commutative irng. We prove here
it is also the case for a free rng on finitely many idempotents and for
a finite irng. A relation to the Wiegold problem for perfect groups is discussed.
\end{abstract}
\maketitle
\section{Introduction}
The Wiegold problem (FP14 in~\cite{openproblems}) is a longstanding
problem in group theory.
It asks whether every finitely generated perfect group $G$ (i.e., $G=[G,G]$)
is singly generated as a normal subgroup.
(We note that it makes no difference if one replaces ``finitely generated''
with ``finitely presented'' in the statement.)
Wiegold answered it positively in the finite group
case, see~\cite[4.2]{Lennox-Wiegold} for a stronger fact.
In this paper, we address a similar problem for rngs.
Recall that a \emph{rng} (a.k.a.\ a pseudo-ring) is an algebraic structure
satisfying the same properties as a ring, except that it may not have
a multiplicative unit. Every rng $R$ is an ideal of a ring, say,
the unitization $R^+$ of $R$.
We call a rng $R$ an \emph{irng} (a shorthand for an idempotent rng)
if it satisfies $R^2=R$.
Here $R^2=\spann\{ xy : x,y\in R\}$.
The \emph{weight} of a rng $R$ is defined as
\[
\wt(R)=\min\{ n : \mbox{$R$ is generated by $n$ elements as an ideal}\}.
\]
Obviously any unital rng is an irng with weight one.
We also notice that $\wt(R)$ is at least the minimal number of generators
of the additive group $R/R^2$.
The direct sum $R=\bigoplus\IZ$ of infinitely many copies
of $\IZ$ is an example of irng with $\wt(R)=+\infty$, but not finitely generated.
In this paper, we address the following.

\begin{prob}
Find a finitely generated irng whose weight is larger than one;
more generally, whose weight is larger than any prescribed positive integer.
\end{prob}

Even without the finite generation assumption, there seems no known example of
irng $R$ with $1<\wt(R)<+\infty$.
For a subset $Z$ in a rng $R$, we denote by $\iip{Z}$ the ideal generated by $Z$.
We note that if an irng $R$ satisfies $R=\iip{Z}$, i.e., $R=\spann(Z+RZ+ZR+RZR)$,
then because $R^3=R$ it actually satisfies $R=\spann RZR$.
\section{Irngs with weight one}
The following result is well-known, but we include the proof for the reader's convenience.
See Theorem~76 in~\cite{kaplansky} for more general fact.
\begin{thm}
Every finitely generated commutative irng $R$ has a unit and hence $\wt(R)=1$.
\end{thm}
\begin{proof}
Let $x_1,\ldots,x_n$ be generators and $\pmb{x}=[x_1,\ldots,x_n]^T\in M_{n,1}(R)$.
Since $R^2=R$, there is $A\in M_n(R)$ such that $A\pmb{x}=\pmb{x}$.
Let $d$ be the determinant of $I-A$, considered in the unitization $R^+$.
Then, $z=1-d \in R$. By Cramer's formula, one has
\[
\mathrm{diag}(d,\ldots,d)\,\pmb{x}=(I-A)^{\sim}\,(I-A)\pmb{x}=\pmb{0},
\]
where $(I-A)^{\sim}$ is the adjugate matrix (the transpose of the matrix of cofactors).
This means that $zx_i=x_i$ for all $i$, and $z$ is a unit for $R$.
\end{proof}

\begin{thm}\label{main}
Let $R$ be a rng generated by $\{x_1,\ldots,x_n\}$ as an ideal and assume that
for every $i$ there is $u_i\in\iip{x_i, \ldots, x_n}$ such that either $x_i=x_iu_i$ or $x_i=u_ix_i$.
Then, $\wt(R)=1$.
\end{thm}

\begin{proof}
Define elements $w_1, \ldots, w_n$ in $1+R\subseteq R^+$ by setting $w_1=1-u_1$ and for $i\geq 2$
$$w_i =
\begin{cases}
w_{i-1}(1-u_i) &\text{if $x_i=u_i x_i$,}\\
(1-u_i) w_{i-1} &\text{if $x_i=x_i u_i$.}\\
\end{cases}
$$
Let further  $z_i=1-w_i\in R$; we claim that $R=\iip{z_n}$.

\smallskip
Notice first that for all $i$ we have $x_i = z_i x_i$ or $x_i=x_i z_i$; in any case, $x_i\in\iip{z_i}$. It follows that $u_i\in\iip{z_i, \ldots, z_n}$. Since $z_i-z_{i-1}$ simplifies (for $i\geq 2$) to $u_i w_{i-1}$ or to $w_{i-1}u_i$, we further deduce that $z_{i-1}$ is in $\iip{z_i, \ldots, z_n}$. Thus,
$$\iip{z_n}=\iip{z_1, \ldots, z_n}=\iip{x_1, \ldots, x_n}=R$$
as claimed.
\end{proof}

We have the particular case mentionned in the abstract by setting $u_i=x_i$:

\begin{cor}
The free rng on finitely many idempotents is singly generated as an ideal.\qed
\end{cor}

We state and prove two further corollaries of Theorem~\ref{main}.

\begin{cor}
Every finite irng $R$ is generated by idempotents as an ideal and $\wt(R)=1$.
\end{cor}
\begin{proof}
Let $R$ be a finite irng and $I$ be the ideal generated by the idempotents in $R$.
Suppose by contradiction that $I\neq R$.
Then, $R/I$ is a finite irng and hence contains a non-nilpotent element (e.g.,\ by Levitzki's theorem).
Let $a\in R$ be a pre-image of a non-nilpotent element in $R/I$.
Since $R$ is finite, $a^m=a^n$ for some $m\neq n$. Since $a^{m^k}=a^{n^k}$ for every $k$,
we may assume that $a^m=a^n$ with $m>2n$. Then, $a^{m-n}$ is an idempotent in $R\setminus I$.
A contradiction. This proves the first half, and the second half follows from the previous theorem.
\end{proof}

\begin{rem}We note that a finite irng need not be unital; e.g.,
$R=\left(\begin{smallmatrix} *&*&*\\0&0&*\\0&0&*\end{smallmatrix}\right)\subset M_3(\IZ/2\IZ)$.
\end{rem}

Smoktunowicz and later Bergman have independently informed us about the following corollary of Theorem 2. With the kind permission of George Bergman we include his argument. Recall that a semigroup $S$ is called idempotent if $S^2=S$. We do not assume that a semigroup contains an identity element.

\begin{cor}  \label{C-bergman} Let $S$  be an idempotent semigroup which is finitely generated
as a semigroup, or more generally, as a left ideal, and  let $k$ be a unital
commutative ring. The semigroup algebra  $R = k S$  is
generated as a $2$-sided ideal by a single element.
\end{cor}

Note that $S$ is idempotent and generated by a finite set $X \subset S$ as a left ideal if and only if $S=SX$. For the proof we need the following lemma.

\begin{lem} \label{L-bergman}
Let  S  be a semigroup containing a finite set  X  such that $S = S X.$
Then  $S$  is generated as a 2-sided ideal by a subset  $X_0$  of
$X$  with the property that every  $x\in X_0$  satisfies  $x\in S x S x$.
\end{lem}
\begin{proof} First, some notation. For  $x, y \in S$,  we shall write  
$y <_1 x$  if  $y \in S x$. Clearly,  the binary relation  $<_1$  is transitive.
Thus, our assumption says $ \forall y\in S\  \exists x\in X :  y <_1 x.$

Hence, starting with any  $y\in S$,  we can find an infinite chain
of elements of  $X$  going upward from it with respect to  $<_1$.
Since  $X$  is finite, this chain must involve a repetition; and applying
transitivity both to the chain from  $y$  to an element  $x$  which is
repeated, and to the chain from that  $x$  to a repetition of  $x$,  we get
$\forall y\in S\  \exists x\in X:  y <_1 x <_1 x.$
Hence, letting $X_1 = \{x\in X | x <_1 x\}$, we see that $S = S X_1.$

In the remainder of this proof, for convenience in writing
the ideal generated by an element, we shall write  $S^+$  for
the monoid obtained by adjoining an identity element to  $S$.
We now define a second relation: For  $x,y\in S$,  we shall write  $y <_0 x$  if  $y\in S^+ x S^+ y$. Again, the binary relation  $<_0$  is transitive. Indeed, if $y \in S^+xS^+y$ and $x \in S^+ z S^+x$, then $y \in S^+xS^+y \subset S^+(S^+ z S^+x)S^+y \subset S^+ z S^+y,$ showing that $y <_0 x <_0 z$ implies $y<_0 z$.

Note that for  $y\in X_1$,  if we take an element of  $z \in S$  witnessing
the relation  $y\in S y$, i.e. $y=zy$, then $z$ will have a
right factor in  $X_1$, i.e.\ $z=z'x$ with $z' \in S^+$ and $x \in X_1$. We conclude $y \in S^+xS^+y$ for some $x \in X_1$ and thus
$\forall y\in X_1 \ \exists x\in X_1:  y <_0 x.$

%(Remark:  In the derivation of (11), (8) has taken the place
%that (1) had in the derivation of (5).  (8) in fact shows that
%we could choose  x  satisfying a conclusion stronger than that
%of (9):  y\in S^1 x y.  But if we had defined  <_0  by that
%stronger condition, we would not have had (10).)

From finiteness of  $X_1$ and a similar argument as above,  we get:
$\forall y\in X_1\  \exists x\in X_1:  y <_0 x <_0 x$. We set
$X_0 = \{x\in X_1 | x <_0 x\}.$ In particular, for every $x \in X_0$, we have $x \in SxSx$. Moreover, for all $x \in X_1$, $x \in S^+X_0S^+x \subset S^+X_0S$ and hence $X_1 \subset S^+X_0S$. We conclude
$S=SX_1 \subset SX_0S.$ This finishes the proof.
\end{proof}

%(Still more generally, this is true of any ring  R  which is generated
%as a 2-sided ideal by a multiplicative subsemigroup  S  that is
%idempotent and finitely generated as a left ideal.)

\begin{proof}[Proof of Corollary~\ref{C-bergman}] Apply Theorem~\ref{main} to the set  $X_0$  given
by Lemma~\ref{L-bergman}.
\end{proof}

%\begin{rem}
%Note that Corollary~\ref{C-bergman} covers the case where $R$ is \emph{idempotent} in the strong sense, that every element of $R$ is a product (rather than just a sum of such) of elements in $R$. Indeed, $(R,\cdot)$
%\end{rem}

%On the other hand, the following example shows that the same reasoning
%does not apply to semigroup rings of semigroups that are merely
%finitely generated as \emph{2-sided} ideals.
%
%Lemma 2.  In the monoid with presentation
%
%(16)	M = < r, s_1, s_2, t  |  r s_i s_i t = s_i  (i=1,2) >,
%
%the ideal
%
%(17)	S = M s_1 M \cup M s_2 M
%
%is idempotent, and is finitely generated as an ideal (of  M,  and
%hence, by idempotence, of itself), but is not generated as an ideal
%by  \{x\in S | x\in (Sx \cup xS)\}  (rather, the latter set is empty).
%S  is also not generated (as an ideal of  M,  equivalently, as an
%ideal of itself) by a single element.

%Incidentally, the index of the latest (17th) edition of the Kourovka
%Notebook of open questions in group theory shows Wiegold as proposer
%of 17 questions that are still open, and 9 that were shown in earlier
%versions but have since been solved.  I think the one you refer to
%is number 5.52.  (My copy is in my office, and I'm at home.)

%
\section{A relation to the Wiegold problem}
%For a group $G$, let $I(G)=\ker(\e\colon\IZ G \to\IZ)$ be its augmented ideal.
%Then, $I(G)/I(G)^2$ is the abelianization of $G$.
%Thus  $I(G)$ is a finitely generated irng iff $G$ is a finitely generated perfect group.
%We note that one has $I(G)=\iip{\delta_g-\delta_1}$ iff $G=\iip{g}$.
%
Let $R$ be a rng and $n$ be a positive integer. For $1\le i\neq j\le n$ and $r\in R$,
we denote by $E_{i,j}(r)$ the elementary matrix with $1$'s on the diagonal,
$r$ in the $(i,j)$-th entry, and $0$'s everywhere else.
Thus, $E_{i,j}(R)\in\GL_n(R^+)$, where $R^+$ is again the unitization of $R$.
We define $\EL_n(R)$ to be the subgroup of $\GL_n(R^+)$ generated by
all the elementary matrices $E_{i,j}(r)$.
They satisfy the Steinberg relations:
\begin{itemize}
\item $E_{i,j}(r) E_{i,j}(s)=E_{i,j}(r+s)$;
\item $[E_{i,j}(r),E_{j,k}(s)]=E_{i,k}(rs)$ if $i\neq k$;
\item $[E_{i,j}(r), E_{k,l}(s)]=1$ if $i\neq l$ and $j\neq k$.
\end{itemize}
Recall that the weight of a group $G$ is defined as
\[
\wt(G)=\min\{ n : \mbox{$G$ is generated by $n$ elements as a normal subgroup}\}.
\]
The following theorem relates the Wiegold problem to the irng problem.

\begin{thm}\label{thm:el}
Let $R$ be a finitely generated irng and $n\geq 3$.
Then $\EL_n(R)$ is a finitely generated perfect group such that
$\wt(R)/n^2 \le \wt(\EL_n(R)) \le \lceil 2\wt(R)/(n^2-n-2) \rceil$.
\end{thm}
\begin{proof}
That $\EL_n(R)$ is finitely generated and perfect follows from the Steinberg relations.
First suppose that $\EL_n(R)$ is generated by $A_1,\ldots,A_m$ as a normal subgroup.
Then, for the unit matrix $I$, one has $A_i-I \in M_n(R)$.
Thus, the collection $Z$ of all entries of $A_i-I$, $i=1,\ldots,m$ is
a subset of $R$ whose cardinality is at most $mn^2$.
Since all $A_i$'s are killed by the canonical homomorphism from
$\EL_n(R)$ onto $\EL_n(R/\iip{Z})$, one sees that $R=\iip{Z}$.
This proves $\wt(R)\le \wt(\EL_n(R))n^2$.
To prove the other inequality, consider the upper triangular matrix
$A=[a_{i,j}]\in\EL_n(R)$ with $1$'s on the diagonal and $a_{i,j}=0$ for $i>j$
and for $(i,j)=(1,n)$. Thus $A$ can have $(n^2-n-2)/2$ many
non-zero entries from $R$.
For every $1\le i,j<n$ and $s,t\in R$, one has
\[
\iip{A}\ni [E_{2,i}(s), [A,E_{j,n}(t)]]
=[E_{2,i}(s), \prod_{k=1}^n E_{k,n}(a_{k,j}t)]
=E_{2,n}(sa_{i,j}t).
\]
Similarly, $E_{n,2}(sa_{i,j}t)\in\iip{A}$ for every $1< i,j\le n$ and $s,t\in R$.
Thus by the Steinberg relations and the fact that $R^2=R$, one has
$E_n(RZR)\subset \iip{A}$ for $Z=\{ a_{i,j} \}$.
This completes the proof.
\end{proof}

Hence, one would solve the Wiegold problem if one finds
a finitely generated irng $R$ with $\wt(R)>9$.
In this regard, the irng problem is harder than the Wiegold problem,
but the authors feel that rngs may be more tractable than groups.

\vspace{0.2cm}

There is another connection to the Wiegold problem. Let $G$ be a group, $k$ be a finitely generated unital ring, and 
let $\omega_kG  \subset k G$ be the augmentation ideal inside the 
group ring of $G$ with coefficients in $k$. 
Since $\omega_k(G)/\omega^2_k(G) = k \otimes_\IZ G_{ab}$, $\omega_k(G)$ 
is a finitely generated irng in many cases, e.g.\ if $G$ is perfect. 
It is also clear that $w(\omega_k(G)) \leq w(G)$. 
Note that this inequality can be strict. Indeed, Theorem~\ref{main} 
shows that $w(\omega_{\IZ/3\IZ}(\IZ/2\IZ \ast \cdots \ast \IZ/2\IZ))=1$ 
since $\omega_{\IZ/3\IZ}(\IZ/2\IZ \ast \cdots \ast \IZ/2\IZ)$ is generated 
by finitely many idempotents. 
On the other side $w(\IZ/2\IZ \ast \cdots \ast \IZ/2\IZ)$ equals the number 
of factors in the free product. We do not know of any example where 
the inequality is strict if $R=\IZ$.

It is natural to ask whether a finitely generated irng must have few generators as a left ideal in itself. In this case, more can be said using $\ell^2$-invariants. Recall that 
$ b^{(2)}_1(G)$ -- a numerical invariant of $G$ ranging in the interval $[0,\infty]$ -- denotes the first $\ell^2$-Betti number of $G$

\begin{thm}
Let $G$ be a group. Then, $\omega_{\IZ}(G)$ needs at least $\lceil b^{(2)}_1(G)+1 \rceil$ generators as a left ideal in itself.
\end{thm}
\begin{proof}
If $G$ is finite, then $b^{(2)}_1(G)=0$. Thus, we may assume that $G$ is infinite.
By~\cite{PT}, L\"uck's dimension of the space $Z^1(G,\ell^2 G)$ of $1$-cocycles on $G$ with values in the Hilbert space $\ell^2 G$ -- endowed with the $\IZ G$-module structure given by the left regular representation -- is equal to $b^{(2)}(G)+1$. For more information on $\ell^2$-Betti numbers, L\"uck's dimension function and the left regular representation, see~\cite{PT} and the references therein.

It is well known that $Z^1(G,\ell^2 G)= \hom_{\IZ G}(\omega_{\IZ}(G),\ell^2 G)$, where $\hom_{\IZ G}$ denotes the set of left-module homomorphisms. If $S \subseteq \omega_{\IZ}(G)$ is a set of generators of $\omega_{\IZ}(G)$ as a left-ideal in itself. It follows that $\hom_{\IZ G}(\omega_{\IZ}(G),\ell^2 G) \subseteq \ell^2 G^{\oplus S}$ by evaluation. Since $\dim_G \ell^2 G^{\oplus S} = |S|$, the claim follows.
\end{proof}

There are examples of finitely generated perfect groups with first $\ell^2$-Betti number as large as we wish; any free product of non-trivial perfect groups works. Thus, for any positive integer $k$, a finitely generated irng can be found which needs at least $k$ elements to generate it as a left ideal.

\subsection*{Acknowledgment}
This research was carried out while the authors were visiting at
the Institut Henri Poincar\'e (IHP) for the Program
``von Neumann algebras and ergodic theory of group actions.''
The authors gratefully acknowledge the kind hospitality and
stimulating environment provided by IHP and the program organizers.
The authors would like to thank Professor Agata Smoktunowicz and
Professor George Bergman for their help regarding the irng problem,
and Professor Martin Kassabov for improving the estimate in Theorem~\ref{thm:el}.

\end{document}